\documentclass[12pt,a4paper]{amsart}
\usepackage{amssymb,amsmath,amsthm}
\usepackage{graphicx}
\usepackage{xcolor}

\usepackage{epsfig}

\usepackage[english]{babel}
\usepackage[utf8]{inputenc}
\usepackage{amsmath}
\usepackage{graphicx}
\usepackage{amsfonts}
\usepackage{enumitem}
\usepackage{amssymb}
\usepackage[colorinlistoftodos]{todonotes}
\usepackage{geometry}
\usepackage{amsthm}
\usepackage{enumitem} 
\usepackage{hyperref}

\newcommand{\remove}[1]{}
\usepackage{tikz}

\long\def\onefigure#1#2{
\begin{figure*}[tbp]
\begin{center}
#1
\end{center}
\caption{#2}
\end{figure*}
} 

\newcommand{\lipefig}[2]  
{\onefigure{\mbox{\psfig{file=#1.eps}}}{\label{f:#1} #2} }

\newcommand{\EE}{\mathbb{E}}

\newcommand{\RR}{\mathbb{R}}

\newcommand{\NN}{\mathbb{N}}
\newcommand{\ZZ}{\mathbb{Z}}

\newtheorem{thm}{Theorem}[section]
\newtheorem{lem}[thm]{Lemma}

\newtheorem{claim}[thm]{Claim}
\newtheorem{fact}[thm]{Fact}

\theoremstyle{definition}

\DeclareMathOperator{\conv}{conv}

\DeclareMathOperator{\vol}{vol}
\DeclareMathOperator{\Area}{Area}

\numberwithin{equation}{section}

\begin{document}

\title{The integer hull of the set $\{(x,y)\in \RR^2: xy\ge N\}$}
\author{Antal Balog and Imre B\'ar\'any}

\maketitle

\begin{abstract}: The integer convex hull $I(H_N)$ of the set $H_N=\{(x,y)\in \RR^2: xy\ge N\}$ is the convex hull of the lattice points in $H_N$. The vertices of $I(H_N)$ lie in the square $[1,N]^2$. Improving on a recent result of  Alc\'antara et al. ~\cite{Santos} we show that the number of vertices of $I(H_N)$ is of order $N^{1/3}\log N$. We also show that the area of the part of $H_N \setminus I(H_N)$ that lies in the square $[1,N^{2/3}]^2$ is also of order $N^{1/3}\log N$.
\end{abstract}

 \bigskip
\section{Introduction}
Define $H_N=\{(x,y)\in \RR^2: xy\ge N, x,y>0\}$ which is an unbounded convex set and $N$ is a (large) positive integer. Its {\sl integer convex hull} or {\sl integer hull} for short is the convex hull of the integer (or lattice) points in $H_N$:
\[
I(H_N)=\conv(\ZZ^2 \cap H_N).
\]
This is bounded by the two halflines $\{(x,1): x\ge 1\}$ and $\{(1,y): y\ge 1\}$ and a convex lattice chain connecting $(1,N)$ with $(N,1)$. Our target is to determine the order of magnitude of the number of vertices of $I(H_N)$. The traditional notation for the number of vertices of a polytope $P\subset \RR^d$ is $f_0(P)$, so we are interested in $f_0(I(H_N))$.

\medskip
Motivation for this question comes from integer programming and, more generally, from the theory of geometry of numbers. The determination or estimation of $f_0(I(H_N))$ is considered in a recent paper by Alc\'antara et al. \cite{Santos}. They establish among other results that
\begin{equation}\label{eq:Sant}
N^{1/3} \ll f_0(I(H_N)) \ll N^{1/3} \log N,
\end{equation}
and ask what the right order of magnitude of $f_0(I(H_N))$ is. Here and throughout the paper the notation $f(N)\ll g(N)$ means that there is a constant $C>0$ such that $0<f(N)<Cg(N)$ for every $N \in \NN$. The notation $f(N)\gg g(N)$ is analogous.  We will also use the $O(.)$ notation. Our main result establishes the order of magnitude of $f_0(I(H_N))$. 

\begin{thm}\label{th:main}
\[
N^{1/3} \log N \ll f_0(I(H_N)) \ll N^{1/3} \log N.
\]
\end{thm}

A similar question is considered in our earlier paper \cite{BalBar}. Writing $B^2$ for the Euclidean unit disk centered at the origin the set $P_r=I(rB^2)$ is a convex lattice polytope, the integer hull of $rB^2$. It is shown in \cite{BalBar} that 
\begin{equation}\label{eq:BalB}
 0.33 r^{2/3} < f_0(P_r) < 5.55 r^{2/3}.   
\end{equation}
Most likely the limit $r^{-2/3}f_0(P_r)$ exists. Balog and Deshoullier \cite{BalD}
proved that the average of $r^{-2/3}f_0(P_r)$ in the interval $[R,R+M]$ tends to an explicit constant 3.453.. as $R \to \infty$ and $M \to\infty$ (for instance $M=\log R$). The upper bound in (\ref{eq:BalB}) is easier and follows from Andrews theorem (Theorem~\ref{th:Andr} below). The lower bound is more difficult and the method of its proof will be used here with suitable modifications. But the case of hyperbola is more involved because the curvature changes from $N^{-1/2}$ to $N^{-2}$ while it is constant for the disk.

\medskip
We observe that $I(H_N)$ has few, about $2N^{1/3}$ vertices outside a square $[1,N_2]^2$, where $N_2$ is about $N^{2/3}$, because we can actually list these vertices. They are
\[
(x,y) = (k,\lceil\frac Nk\rceil),\quad 1\leq k\leq N^{1/3},
\]
and by symmetry
\[
(x,y) = (\lceil\frac Nk\rceil,k),\quad 1\leq k\leq N^{1/3}.
\]
Indeed, the integer point $(1,N)$ is a vertex and $(k,\lceil\frac Nk\rceil)$ is certainly a vertex whenever $(k-1,\lceil\frac N{k-1}\rceil)$ is a vertex and
\[
\frac N{k+1} > 2\lceil\frac Nk\rceil - \lceil\frac N{k-1}\rceil. 
\]
This inequality follows if $k\leq N^{1/3}$. We mention that the above list of vertices is practically the lower bound of Alc\'antara et al. ~\cite{Santos}. For later use we point out that with
\begin{equation}\label{eq:N_12}
    N_1=\lfloor N^{1/3}\rfloor =N^{1/3}+O(1),\quad N_2=\lceil \frac N{N_1} \rceil=N^{2/3}+O(N^{1/3}),
\end{equation}
the points $(N_1,N_2)$ and $(N_2,N_1)$ are vertices of $I(H_N)$.
Introducing the notation $Q_N=H_N\cap [1,N_2]^2$ another form of Theorem~\ref{th:main} is the following.

\begin{thm}\label{th:main2}
\[
N^{1/3} \log N \ll f_0(I(Q_N)) \ll N^{1/3} \log N.
\]
\end{thm}

Our second main result is about how well $I(H_N)$ approximates $H_N$ when the approximation is measured by the area $H_N$ missed by $I(H_N)$. But this area is infinite because $H_N \setminus I(H_N)$ contains a large part of the set $\{(x,y)\in \RR^2: x\in [0,\infty), y\in [0,1)\}$. So it is better to consider $(H_N \setminus I(H_N))\cap [1,N]^2$. The area of this set is of order $N$; see the remark at the end of Section~\ref{sec:prepa}. So we rather work with $Q_N$ again and define
\[
A_N=\Area(Q_N \setminus I(Q_N)).
\]

\begin{thm}\label{th:area}
\[ N^{1/3} \log N \ll A_N \ll N^{1/3} \log N.\]
\end{thm}

\bigskip
\section{Connections to lattice polytopes and random polytopes}

The theorems mentioned in this section serve to support our motivations, but are not explicitly used in the proof of our main results.

\medskip
A lattice polytope $P$ in $\RR^d$ is a convex polytope whose vertices belong to the lattice $\ZZ^d$. A famous theorem of G. E. Andrews \cite{Andr} says the following.

\begin{thm}\label{th:Andr} For a lattice polytope $P \subset \RR^d$ with $\vol P>0$
\[
f_0(P) \ll (\vol P)^{\frac {d-1}{d+1}}
\]
with the implied constant depending only on $d$.
\end{thm}

\medskip
Write $P_r^d$ for the integer convex hull of $rB^d$ where $B^d$ is the Euclidean unit ball in $\RR^d$ centered at the origin. Extending the result $(\ref{eq:BalB})$
B\'ar\'any and Larman \cite{BarLar98} proved that
\[
r^{\frac{d(d-1)}{d+1}} \ll f_0(P_r^d) \ll r^{\frac{d(d-1)}{d+1}}.
\]
This shows the exponent in Andrews's theorem is best possible. 

\medskip
As $I(Q_N)$ is a lattice polygon whose area is of order $N^{4/3}$ Andrews's theorem implies that $f_0(I(Q_N)) \ll N^{4/9}$, larger than the bound $N^{1/3}\log N$ in Theorem~\ref{th:main}. This is partially explained by the analogy between lattice polytopes and random polytopes, that we describe below.

\medskip
Let $K \subset \RR^d$ be a convex body (compact convex set with nonempty interior) and let $X_n$ be a random sample of $n$ points chosen from $K$ independently and with the uniform distribution. Then $K_n=\conv X_n$ is a {\sl random polytope} inscribed in $K$.

\medskip
Assuming $K$ is large and contains $n=|K\cap \ZZ^d|$ lattice points and $n$ is also large, the integer hull $I(K)$ and the random polytope $K_n$ behave quite similarly. This is not a rigorous statement; this is just a guiding principle of what to expect.

\medskip
Random polytopes have been studied for a long time, starting with Sylvester's four-point problem~\cite{Syl}. Their systematic study began with R\'enyi and Sulanke \cite{RenSul}. What we need here are some results on the expectation 
$\EE f_0(K_n)$ and on the missed volume, that is, on $\vol (K\setminus K_n)$, to be denoted by $E(K,n)$. We assume $\vol K=1$, in which case the probability measure
on $K$ coincides with the Lebesgue measure and $d\ge 2$. The first result we need is Efron's identity \cite{Efron}:

\begin{thm}\label{th:efron} For every convex body $K\subset \RR^d$ of volume one
\[
\EE f_0(K_n)=nE(K,n-1). 
\]
\end{thm}

The function $v: K \to \RR$ is defined by 
\[
v(x)=\min\{\vol(K\cap H): x\in H \mbox{ and }H \mbox{ is a halfspace}\}.
\]
Note that the definition of $v$ also makes sense for unbounded convex sets. Here, of course, $K\cap H$ is a cap of $K$ cut off by the halfspace $H$. The level sets of $v$ are defined as usual: $K(v\le t)=\{x\in K: v(x)\le t\}$. This set is called the {\sl floating body} of $K$ with parameter $t$, and is convex because it is the intersection of halfspaces. The set $K(v\le t)$ is the {\sl wet part}. The name comes from the picture when $K$ is a 3-dimensional convex body containing $t$ units of water. It is known (see for instance \cite{Leicht}, \cite{BarLar88} and \cite{SchutW}) that when $K$ has smooth enough boundary, its Gauss-Kronecker curvature $\kappa$ is positive everywhere and $\vol K=1$, then the volume of the wet part satisfies 
\begin{equation}\label{eq:wetp}
\vol K(v\le t)=\mbox{const}(d)\left(\int\kappa^{\frac1{d+1}}dx\right)t^{\frac{2}{d+1}}(1+o(t))
\end{equation}
where the integration goes on the boundary of $K$ according to the surface area.

\medskip
The second result that we need is due to B\'ar\'any and Larman \cite{BarLar88}.

\begin{thm}\label{th:wetp} For every convex body $K\subset \RR^d$ of volume one
\[
\vol K(v\le 1/n) \ll E(K,n) \ll \vol K(v\le 1/n),
\]
where the implied constants only depend on $d$.
\end{thm}
The parameter $t=1/n$ in this theorem corresponds to $t=\frac 1n \vol K$ when $\vol K$ is different from 1.

\medskip
Returning now to $Q_N$ we see that $Q_N$ contains about $n \approx N^{4/3}$ lattice points. With $K=Q_N$ and $v=v_K=v_{Q_N}$ we can determine the area of the wet part $K(v\le t)$ when $t=\frac 1n \Area K\approx 1$. Using formula (\ref{eq:wetp}) (or direct integration) gives
\[
N^{1/3}\log N \ll \Area Q_N(v\le 1) \ll N^{1/3}\log N. 
\]

Then for a random polytope $K_n$ inscribed in $Q_N$ with $n\approx N^{4/3}$ random points one would expect that $f_0(K_n)$ is about $N^{1/3}\log N$ and its missed area is about $N^{1/3}\log N$, too. If the analogy between random polytopes and lattice polytopes works, the same number of vertices for $I(Q_N)$ and the same missed area $A_N$ would be expected. We will see that this is indeed the case. 

\medskip
Yet, a word of warning is in place here: in both Theorems \ref{th:efron} and \ref{th:wetp} $K$ is a fixed convex body and $n$ goes to infinity. But in our case both $K=Q_N$ and $n\approx N^{4/3}$ depend on $N$. Still, we can determine the expectations of the number of vertices and the missed area of the random polytope $K_n$ when $K=Q_N$ and $n=N^{4/3}$. The outcome is that both $\EE f_0(K_n)$ and $E(K,n)$ are of order $n^{1/4} \log n$ or $N^{1/3}\log N$, as expected. We will not give the proof of these results here. \qed

\bigskip
\section{Proof of Theorem \ref{th:main}, the upper bound}

As it turns out, the vertices of $I(H_N)$ lie in a narrow strip above the hyperbola $xy=N$. This is the content of the following lemma.

\begin{lem}\label{l:narrow}If $(x,y)$ is a vertex of $I(H_N)$, then $N \le xy \le N+2N^{1/3}$.
\end{lem}
 {\bf Proof.} This is based on Minkowski's theorem on lattice points in $0$-symmetric convex bodies. Denote by $H_N^*$ the reflection of $H_N$ with respect to the point $(x,y)$. The set $K=H_N \cap H_N^*$ is convex, centrally symmetric with respect to the point $(x,y) \in \ZZ^2$, and contains no lattice point apart from $(x,y)$ because $(x,y)$ is a vertex of $I(H_N)$. By Minkowski's theorem $\Area K \le 4$ and a straightforward (and generous) computation show that $xy \le N+2N^{1/3}$. 

\medskip
According to this lemma, every vertex $(x,y)$ of $I(H_N)$ lies on a hyperbola $xy=n$ with $n\in [N,\ldots,N+2N^{1/3}]$. The number of lattice points on this hyperbola is equal to the number of divisors $d(n)$ of $n$. It is known (see for instance \cite{HarW}) that
\[
\sum_{n\le N}d(n)=N \log N +(2\gamma -1)N +O(N^{1/3}), 
\]
where $\gamma=0.57721..$ is the Euler constant. The better error term $O(N^{0.315})$ was proved by Huxley~\cite{Hux} but we do not need it. Then 
\[
\sum_{N\le n \le N+2N^{1/3}}d(n)=2N^{1/3}\log N+O(N^{1/3}). 
\]

So the set $\{(x,y)\in \RR^2: N\le xy \le N+2N^{1/3}\}$ contains $2N^{1/3}\log N+O(N^{1/3})$ lattice points and indeed $f_0(I(H_N))\ll N^{1/3} \log N.$ \qed

\medskip
Our Theorem~\ref{th:main} claims that a positive proportion of the integer points in the above hyperbolic strip are actually vertices of $I(H_N)$. For technical reasons, we change to a somewhat narrower strip. In this way, we can miss a few vertices, but we can afford to get a lower bound.    
Set $\Delta=\frac12 N^{1/3}$, and write $H^1$ and $H^2$ for the hyperbolas with equations $xy=N$ and $xy=N+\Delta$, respectively, and finally $S$ for the strip between $H^1$ and $H^2$. The above argument shows that  
\begin{equation}\label{eq:delta} 
|S \cap \ZZ^2| = \Delta \log N+O(N^{1/3}).
\end{equation}

\medskip
We mention that the proof of the upper bound in Theorem~\ref{th:main} in \cite{Santos} is different. It starts with the observation that, because of symmetry, it suffices to bound the number of vertices $(x,y)$ with $\sqrt N\le x\le N$. Next, it is checked that if $(x,y)$ with $2^kN^{1/2}\le x\le 2^{k+1}N^{1/2}$ is a vertex of $I(H_n)$, then it is also a vertex of the integer hull of the convex set $H_N\cap \{(u,v)\in \RR^2: 2^kN^{1/2}\le u\le2^{k+1}N^{1/2}, N^{1/2}/2^{k+1} \le v \le 3N^{1/2}/2^{k+1}\}$ where $k=0,1,\ldots,k_0$ and $k_0=\lfloor \frac 12 \log N\rfloor$. The area of this convex set is less than $N$ and according to Andrews' theorem (Theorem~\ref{th:Andr}) its integer convex hull has at most $\ll N^{1/3}$ vertices. There are $O(\log N)$ such convex sets that give the upper bound on $f_0(I(H_N))$.

\bigskip
\section{Auxiliary computations}\label{sec:aux}

Let $(a,b)\in \ZZ^2$ be a primitive vector with $0<a, b$, that is $\gcd(a,b)=1$. The line $L$ with equation $bx+ay=k$ intersects $H^1$ and $H^2$ in points with $x$-component $x_1,x_2$ and $x_1^*,x_2^*$, respectively, see Figure~\ref{fig:tang}, and 
\begin{equation}\label{eq:sectons}
x_1,x_2=\frac{k\pm \sqrt{k^2-4abN}}{2b} \mbox { and }  x_1^*,x_2^*=\frac{k\pm \sqrt{k^2-4ab(N+\Delta)}}{2b}  
\end{equation}
where the signs are chosen so that $x_1<x_2$ and $x_1^* <x_2^*$.

\begin{figure}[h!]
\centering
\includegraphics[scale=0.8]{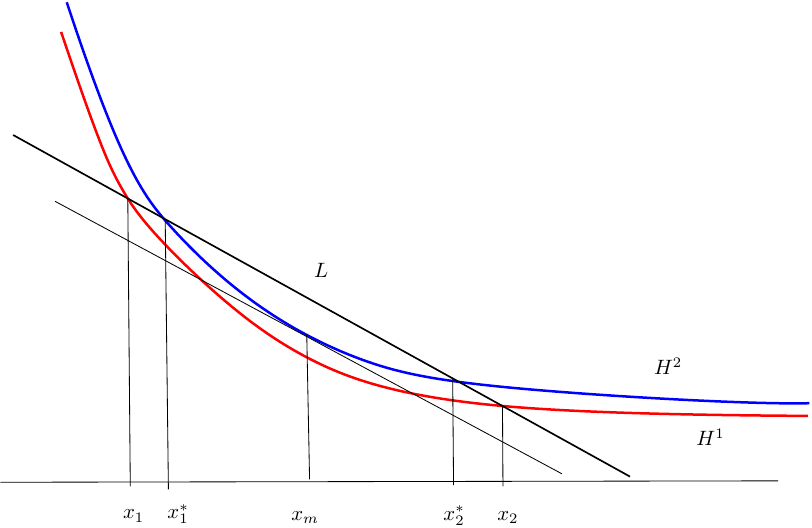}
\caption{ $L$ intersecting $H^1$ and $H^2$.}
\label{fig:tang}
\end{figure}

\medskip
The slope of $L$ is $-m=-b/a$, and $L$ is tangent to $H^2$ iff $k^2=4ab(N+\Delta)$. Then the point of tangency has $x$-component (see Figure~\ref{fig:tang}) 
\begin{equation}\label{eq:tang}
x_m=\frac k{2b}=\sqrt{\frac{N+\Delta}m} \mbox{ so }m=\frac{N+\Delta}{x_m^2}. 
\end{equation}
In this case $L\cap S$ is a segment whose projection to the $x$-axis is the interval $[x_1,x_2]$ with $k^2=4ab(N+\Delta)$ and 
\begin{equation}\label{eq:cor}
x_2-x_1=2 \sqrt{\frac{\Delta}m}=2x_m\sqrt{\frac{\Delta}{N+\Delta}}.
\end{equation}

\medskip
Here $\frac {x_2}{x_1}>1$ is close to one. More precisely
\[
\frac {x_2}{x_1}=1+\frac {x_2-x_1}{x_1}=1+2\frac {x_m}{x_1}\sqrt{\frac{\Delta}{N+\Delta}}<1+2\frac {x_2}{x_1}\sqrt\frac{\Delta}{N}.
\]
This shows that
\begin{equation}\label{eq:ratio}
\frac {x_2}{x_1}<1+2N^{-1/3},
\end{equation}
if $N$ is large enough. Note that this inequality holds for any segment that is entirely in $S$.
\medskip

We will also need a simple fact from multiplicative number theory. Define $F(w)$ as the number of primitive vectors $(a,b)$ with $1\le a, b$ and $ab\le w$. 
\begin{fact}\label{fa:divisors} $$F(w)=\frac 1{\zeta(2)}w\log w +O(w).$$ 
\end{fact}

\medskip

The {\bf proof} is simple. Let $P(w)$ be the set of all primitive vectors $(a,b)$ with $1\le a,b$ and $ab\le w$. Writing $\omega(n)$ for the number of different prime divisors of $n\in\NN$, one can factor $n$ into $n=ab, 1\leq a, b,\, \gcd(a,b)=1$ in $2^{\omega(n)}$ ways. We have 
\begin{align*}
F(w)&=\sum_{(a,b)\in P(w)}1=\sum_{n\le w}\sum_{(a,b)\in P(w), ab=n}1 \\
&=\sum_{n\le w}2^{\omega(m)} =\frac 1{\zeta(2)}w\log w+O(w).
\end{align*}
The last step is a routine computation; see, for example, Chapter 8.3 in \cite{NivMZ}.
Note that $\zeta(2)=\pi^2/6$.
\qed

\bigskip
\section{Proof of Theorem \ref{th:main}, the lower bound}
 
\medskip
We prepare the ground to show that a positive fraction of the points in $S\cap \ZZ^2$ is a vertex of $I(Q_N)$. We trim $S$ by the condition $x,y \leq N_2$, see (\ref{eq:N_12}), and set $S^*=\{(x,y)\in S: x,y\le N_2\}$, $\,S^1=\{(x,y)\in S:  N_2< y\}$, and $S^2=\{(x,y)\in S: N_2 <x\}$. We can equally well define $S^*$ with the condition $N_1\le x$ rather than $y\le N_2$. The difference is that we cut off the tail on the left hand side with a horizontal or a vertical line across the vertex $(N_1,N_2)$. The difference between the domains is covered by the rectangles $[\frac N{N_2},\frac {N+\Delta}{N_2}]\times [\frac N{N_1},\frac {N+\Delta}{N_1}]$ if $(N_1,N_2)\in S$, or $[\frac N{N_2}, N_1]\times [\frac N{N_1},N_2]$ if $(N_1,N_2)\not\in S$ that do not contain integer points but $(N_1,N_2)$. 
  
\begin{figure}[h!]
\centering
\includegraphics[scale=0.9]{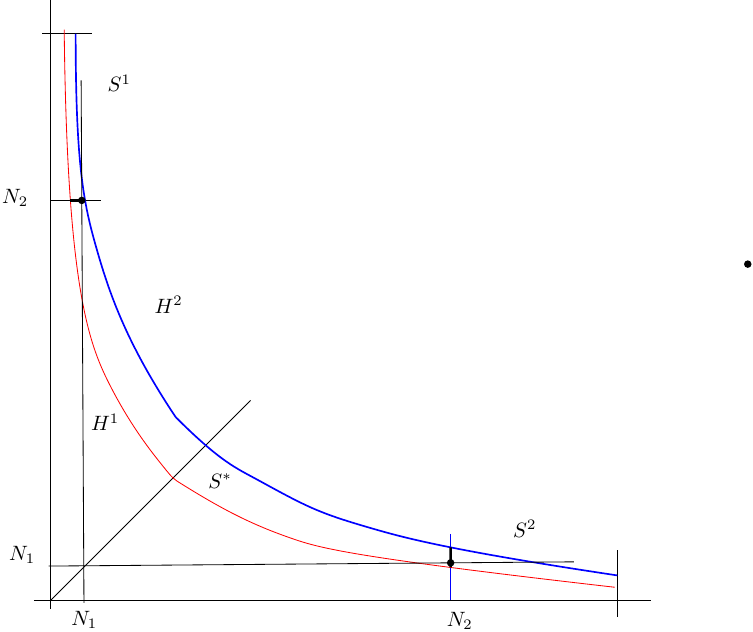}
\caption{$S^*$ and $S^1$.}
\label{fig:H1H2S}
\end{figure}

\medskip
The line $x=k$ intersects $S$ in a segment of length $(N+\Delta)/k-\Delta/k=\Delta/k$ that contains $\lfloor \Delta/k\rfloor$ or $\lfloor \Delta/k\rfloor+1$ lattice points, see Figure~\ref{fig:H1H2S}. Since  
\[
\sum_{k\le N^{1/3}} \frac{\Delta}k=\frac 13 \Delta \log N +O(N^{1/3})
\]
the number of lattice points in $S^1$, and by symmetry in $S^2$, is $\frac 13 \Delta \log N +O(N^{1/3})$. Consequently, 
\[
|S^*\cap \ZZ^2| = \frac 13 \Delta \log N + O(N^{1/3}).
\]

\medskip
We mention in passing that $S^1$ or $S^2$ contains at most $N^{1/3}$ vertices of $I(H_N)$ because there is at most one vertex on a vertical or horizontal line. This implies that the number of vertices of $I(H_N)$ in $S^1$ or $S^2$ {\bf is not} a positive fraction of the lattice points thereby.

\medskip
We let $V$ denote the set of vertices of $I(H_N)$ in $S^*$ and $NV$ the non-vertices, that is, points in $S^*\cap \ZZ^2$ that are not vertices of $I(H_N)$. Note that there are vertices on both lines $x=N_1$ and $x=N_2$, but they can be outside of $S$, because we reduced the width of the hyperbolic strip from $2N^{1/3}$ to $\frac12N^{1/3}$.

\medskip
The projection of $V$ to the $x$ axis splits the interval $[N_1,N_2]$ into $|V|-1$ or $|V|+1$ subintervals, depending on which vertices $(N_1,N_2)$ and $(N_2,N_1)$ are or are not in $V$. In the second case the leftmost and rightmost subintervals are 'one-sided' in the sense that they are limited by the $x$-projection of an element of $V$ only from inside and by $N_1$ or $N_2$ from outside. For any $z\in NV$ we assign a unique $v\in V$ in the following way. We say that $v$ is visible from $z$ (or $z$ is visible from $v$) if the segment connecting them is entirely in $S^*$. Suppose that the $x$ coordinate of $z$ is in the subinterval given by the projections of $v_1$ and $v_2$. One of them must be visible from $z$ because $z$ is not a vertex, and we assign that to $z$. If both of them are visible from $z$, then we assign it to the closest one, if they are equally close, then we assign the one to the right of $z$. We denote this assignment by $z\to v$. Observe that if there are elements of $NV$ with $x$-projection in the leftmost or rightmost one-sided subintervals, then they must be visible from the vertices of $V$ on the inner side, due to convexity. Note also that possibly no $z\in NV$ is assigned to some $v\in V$ at all.
 
\begin{claim}\label{cl:triangl} Assume $z_1,z_2\in NV$, $z_1\to u$ and $z_2\to u$. Then $u,z_1,z_2$ are collinear provided both $z_1$ and $z_2$ are to the left of $u$, or to the right of $u$. In particular, for every $v\in V$ there are at most two halflines starting at $v$ that contain all $z\in NV$ with $z\to v$.
\end{claim}

\medskip
{\bf Proof.} We assume that $z_1,z_2$ are to the left of $u$ and, further, that the longest side of the triangle $T$ spanned by $u,z_1,z_2$ is the one connecting $u$ and $z_1$. Since $[u,z_1]\subset S$, this edge is the longest if both $z_1$ and $u$ are on the hyperbola $H^1$ and the edge is tangent to $H^2$. So, with $z_1=(x_1,y_1)$ and $u=(x_2,y_2)$ we are in the situation described in Section~\ref{sec:aux}. The projection of the triangle to the $x$-axis is of length at most $x_2-x_1$. 

\medskip
The vertical line containing $(x,0)$ intersects $S$ in a segment whose length is  $(N+\Delta)/x-N/x= \Delta/x \le \Delta/x_1$ when $x \in [x_1,x_2]$. It follows, using inequality (\ref{eq:ratio}) that

\[
\Area T \le\frac 12 \frac {\Delta}{x_1}(x_2-x_1) = \frac {\Delta}2 (\frac {x_2}{x_1} -1) < \Delta N^{-1/3} = \frac12.
\]

The area of a lattice triangle is at least $\frac 12$ proving that $z_1,z_2$ and $u$ are indeed collinear. Note that this last step explains $\frac12$ in the choice of $\Delta$. This argument works on the leftmost subinterval as well, while a very similar argument when $z_1,z_2$ are to the right of $u$ covers the other case and the rightmost subinterval as well. \qed

\medskip
This claim ensures that all points of $NV$ can be covered with at most $2|V|$ line segments, each starting from some $v\in V$, ending in some $z\in NV$, with $[z,v]$ lying in $S^*$. Unfortunately, unlike in a convex lattice chain, these segments are not all in different directions. However, all these segments must have a negative slope because the intersection of $S^*$ with any horizontal or vertical line, and also with any line of non-negative slope, is of length at most $\max\{\Delta/N_1,\Delta N_2/N\}<1$, due to convexity and symmetry.

\medskip
In the next step we fix a primitive vector $(a,b)\in \ZZ^2$ with $0<a,b$, and estimate the number of $z\in NV$ such that the segment $[z,v]$, where $z\to v$,  has direction $(a,-b)$. Then $[z,v]$ is contained in a lattice line having equation $bx+ay=k$ with $k\in \ZZ$. Among these lines the one with 
$k_0=\lfloor \sqrt{4ab(N+\Delta)}\rfloor$ is immediately below the hyperbola $H^2$. On a line $bx+ay=k$ with $k\le k_0$ the non-vertices together with $v$ form a segment. Two such segments are coloured blue in Figure~\ref{fig:over}. The projections of these segments to the $x$ axis are disjoint: if two consecutive ones overlap, for instance if in Figure ~\ref{fig:over} we had $z^*\to v_1$ and the segment $[z^*,v_1]$ were blue, then we should have $z^* \to v_2$ instead of $z^* \to v_1$.

\begin{figure}[h!]
\centering
\includegraphics[scale=0.7]{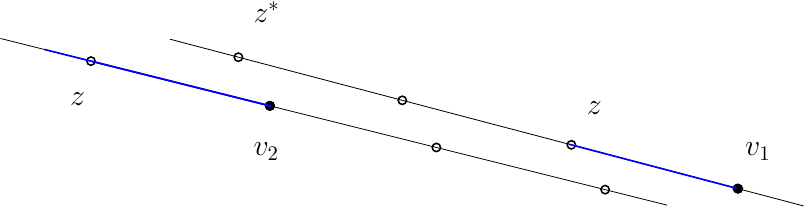}
\caption{$v_1,v_2\in V$ and a few $z \in NV$ with $z\to v_i$.}
\label{fig:over}
\end{figure}

So, according to (\ref{eq:cor}) the total length of the projected segments is at most $x_2-x_1=2\sqrt{\Delta/m}$ where $m$ is given by $-m=-b/a$. So the number, $M_1$, of non-vertices in the given direction below the tangent to $H^2$ satisfies $M_1\le \frac 2a\sqrt{\frac {\Delta}m}= 2 \sqrt{\frac {\Delta}{ab}}$. 

\medskip
The other lattice lines have equation $bx+ay=k_0+k$ with $k=1,2,...$. Their contribution to the non-vertex count is at most (with the notation on Figure ~\ref{fig:tang})
\[
\frac 1a[(x_1^*-x_1)+(x_2-x_2^*)] = \frac {(x_2-x_1)-(x_2^*-x_1^*)}a
\] 
Using equation (\ref{eq:sectons}) gives 
\begin{align*}
(x_2-x_1)-(x_2^*-x_1^*)&=\frac{2\sqrt{k^2-4abN}-2\sqrt{k^2-4ab(N+\Delta)}}{2b}\\
&=\frac{4a\Delta}{{\sqrt{k^2-4ab(N+\Delta)}+\sqrt{k^2-4abN}}}
\end{align*}

Here $k=k_0+h$ with $h=1,2,\ldots,h^*$; we determine $h^*$ later. Summing this for all $h$ gives the upper bound, $M_2$, on the number of non-vertices in the given direction above the tangent line:
\begin{align*} 
M_2& \le \frac{1}{a} \sum_{h=1}^{h^*}\frac{4a\Delta}{\sqrt{(k_0+h)^2-4abN}+\sqrt{(k_0+h)^2-4ab(N+\Delta)}}\\
&= \sum_{h=1}^{h^*}\frac{4\Delta}{\sqrt{k_0^2+2k_0h+h^2-4abN}+\sqrt{k_0^2+2k_0h+h^2-4ab(N+\Delta)}}\\
&\le \sum_{h=1}^{h^*}\frac{4\Delta}{\sqrt{4ab\Delta+2k_0h+h^2}+\sqrt{2k_0h+h^2}}\\
&\le \sum_{h=1}^{h^*}\frac{2\Delta}{\sqrt{4 \sqrt{abN}h}}\le  \sum_{h=1}^{h^*}\frac{\Delta}{\sqrt{h} (abN)^{1/4}}.
\end{align*} 
The sum is taken as long as the last term is at least one. So $\sqrt{h^*}(abN)^{1/4} \le \Delta$ but $\sqrt{h^*+1}(abN)^{1/4} > \Delta$ and so $h^*= \lceil \Delta^2/\sqrt{abN}\rceil.$ Then 
\[
M_2\le \frac {\Delta}{(abN)^{1/4}}\sum_{h=1}^{h^*}h^{-1/2}\le \frac{\Delta^2}{(abN)^{1/2}}\le 2\sqrt{\frac{\Delta}{ab}}.
\]
Here we used the routine estimate $\sum_{h=1}^{h^*}\frac1{\sqrt h}\le 2\sqrt{h^*}$.

\medskip
The number of non-vertices associated with the fixed primitive vector $(a,b)$ is then at most $M_1+M_2\le 4\sqrt{\Delta/ab}$. Suppose $(a_i,b_i)\in \ZZ^2$ for $i=1,\ldots,R$ is the set of primitive vectors (with distinct directions) that take part in the relations $z\to v$ with $z \in NV$ and $v \in V$. Then $R\le 2|V|$ as one vertex is used with at most two vectors $(a_i,b_i).$ At the same time 
\begin{equation}\label{eq:3root}
|NV| \le \sum_{i=1}^R 4\sqrt{\frac{\Delta}{a_ib_i}}=4\sqrt{\Delta}\sum_{i=1}^R \frac1{\sqrt{a_ib_i}}.
\end{equation} 
The last sum is the largest when we choose $R$ different primitive vectors with $a_ib_i$ as small as possible, so with $a_ib_i\le w$ for the least possible $w\in \NN$. From {\bf Fact ~\ref{fa:divisors}} this $w$ satisfies $F(w-1)<R\le F(w)$ and then
\[
R=\frac1{\zeta(2)}w\log w+O(w). 
\]
This $w$ satisfies 
\[ w\ll \frac R{\log R}\ll \frac {|V|}{\log |V|}. \]

\medskip
Finally, by partial summation starting from {\bf Fact ~\ref{fa:divisors}} we get the following.
\begin{align*}
\sum_{i=1}^R \frac1{\sqrt{a_ib_i}}&\le \sum_{(a,b)\in P(w)}\frac1{\sqrt{ab}} = \sum_{n\le w}\frac{2^{\omega(n)}}{\sqrt{n}} = \frac{F(w)}{\sqrt{w}}-\int_1^w F(t)\,dt^{-1/2} = \\& =
\frac{F(w)}{\sqrt{w}}+\frac12\int_1^w F(t)t^{-3/2}\,dt =
\frac2{\zeta(2)}\sqrt{w}\log w + O(\sqrt{w}) \ll\\ &\ll \sqrt{|V| \log |V|}.
\end{align*}
Going back to (\ref{eq:3root}) we see that 
\[
|NV|\ll 4\sqrt{\Delta} \sqrt{|V| \log |V|} \ll  \sqrt{N^{1/3}|V| \log N}.
\]
Finally
\[
N^{1/3}\log N \ll |V|+|NV| \ll |V|+\sqrt{N^{1/3}|V| \log N}
\]
implying that $|V| \gg N^{1/3}\log N$. Thus $f_0(I(H_N))\ge |V|\gg N^{1/3}\log N$.\qed

\bigskip
\section{Preparations for the proof of Theorem~\ref{th:area}}\label{sec:prepa}

\medskip
We begin by setting up the necessary parameters and notation. With the function $y=\frac Nx$ we have 
\[
y'=-\frac N{x^2} \mbox{ and } y''=\frac{2N}{x^3}.
\] 
So the radius of curvature $r=r_x$ at the point $(x,y) \in H^1$ is given by
\[
r_x=\frac{(1+y'^2)^{3/2}}{y''}=\frac{(x^4+N^2)^{3/2}}{2Nx^3}.
\] 
For simpler notation we will write $f(N) \asymp g(N)$ for $f(N) \ll g(N) \ll f(N)$. Because of symmetry, we are going to work only when $x\ge \sqrt N$ so $x^4 \ge N^2$ and $x^4 \le x^4+N^2 \le 2x^4$. Then we have 
\begin{equation}\label{eq:kappa}
\frac {x^3}{2N}\le  r=r_x \le  \frac {\sqrt 2 x^3}{N} \mbox{ or with the new notation } r_x \asymp \frac{x^3}N.
\end{equation}

\medskip
Assume that the primitive vector $p=(a,-b) \in \ZZ^2$ with $0<b\le a$ is the direction of an edge $E_p$ of $I(H_N)$, and let $C^p$ be the corresponding cap cut from $H_N$ by line $L_p$ of $E_p$. The equation of $L_p$ is of the form $bx+ay=k$ for some integer $k$. Let $(x,y)$ be the common point of $H^1$ and the tangent line to $H^1$ parallel with $L_p$. Its equation is $bx+ay=\sqrt{4abN}$ as the computations around equations (\ref{eq:sectons}) and (\ref{eq:tang}) show.

The slope of $L_p$ is given by $y'=\frac{-b}a=-\frac N{x^2}$ and $x=x_p=\sqrt {\frac{aN}b}$ so $x$ is uniquely determined by $p$ and vice versa. 

\medskip
Our main target is to show that $\Area \left(\bigcup C^p \right) \ll N^{1/3}\log N$ where the union is taken for all inner normals $p$ to the edges of $I(Q_N)$. 
This implies $A_N\ll  N^{1/3}\log N$.

\begin{figure}[h!]
\centering
\includegraphics[scale=0.8]{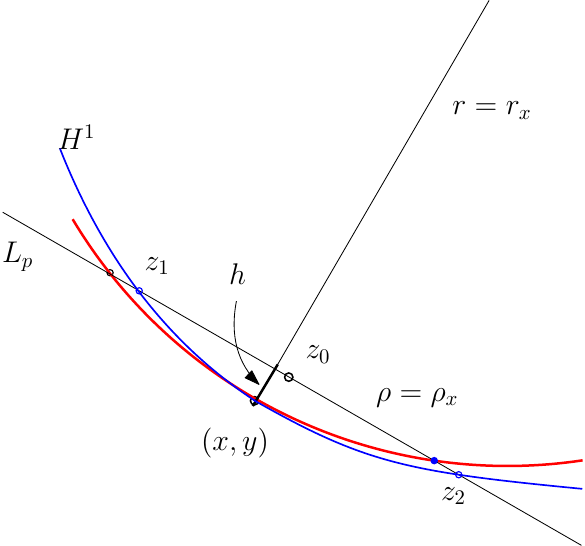}
\caption{$H^1$ and the osculating circle.}
\label{fig:kappa}
\end{figure}
We use the notation of Figure~\ref{fig:kappa} where $H^1$ is coloured blue and the osculating circle red. The height (or width in orthogonal  direction of $p$) of $C^p$ is $h=h_x$. The equation of $L_p$ is $bx+ay=\sqrt{4abN}+h|p|$, where $|p|$ is the Euclidean length of the vector $p$. (We should have written here $\ldots+h|p^{\perp}|$ because $p^{\perp}=(b,a)$ is the inner normal to $I(H_N)$ on $E_p$, but it matters not.) $L_p$ intersects $H_N$ in the segment $[z_1,z_2]$ and the osculating circle in a segment whose length is $2\rho=2\rho_x$. It follows that $\rho^2=h(2r-h)$, or with the $\asymp$ notation $\rho^2\asymp hr$. Let $z_0$ denote the midpoint of the segment $[z_1,z_2]$. The osculating circle approximates $H_N$  in the vicinity of the point $(x,y)\in H^1$.

\begin{claim}\label{cl:rho}When $N$ is large $|z_2-z_0|\asymp \rho$ and $|z_1-z_0|\asymp \rho$.
\end{claim}

\medskip
{\bf Proof.} Set $z_i=(x_i,y_i)$ for $i=0,1,2$. We are going to show $x_2-x_0\asymp \rho.$ This implies that $|z_2-z_0|\asymp \rho$. The proof of $|z_1-z_0|\asymp \rho$ is analogous and is omitted. 
As the interior of $C^p$ is lattice point free, $C^p$ cannot contain a unit disk. Consequently $h<2$. Using formula (\ref{eq:sectons}) we see that
\[
x_2-x_0=\frac1{2b}\sqrt{2h|p|\sqrt{4abN}+h^2|p|^2}, 
\]
and we want to show that this is $\asymp \rho\asymp \sqrt {hr}$, or, after simplifying by $h$ and taking squares, that $|p|(4\sqrt{abN}+h|p|)\asymp b^2r$. 

\medskip
We observe that $|p|\le \sqrt {abN}$. Indeed, taking squares this is the same as $a^2+b^2<abN$ and dividing by $a^2$, this is the same as $1+(b/a)^2<(b/a)N$. Here $b/a\le 1$ and $N/a>2$ obviously. Then using $h\le 2$ we have $4\sqrt{abN}+h|p| \asymp 4\sqrt{abN}$ and the target is to see that 
\[
|p| \sqrt{abN}\asymp rb^2 \mbox{ or by taking squares }(a^2+b^2)abN \asymp r^2b^4.
\]
Here $a^2+b^2\asymp a^2$ (because $0<b\le a$). Dividing by $a^3$ leads to 
\[
N\asymp \left(\frac ba\right)^3r^2\asymp \left(\frac N{x^2}\right)^3\left(\frac {x^3}{N}\right)^2=N.
\]\qed

\medskip
It is easy to determine the area of $C^p$ with the necessary precision: $\Area C^p \asymp h|z_2-z_1|\asymp h\rho$ and since $hr\asymp \rho^2$ we see
\begin{equation}\label{eq:sim}
\Area C^p \asymp h\rho \asymp \frac {\rho^3}r \asymp \sqrt{h^3 r}.
\end{equation}

\medskip
We {\bf remark} here that $\Area \left(\bigcup C^p \right)$ is of order $N$ when the union is taken over all inner normals to $I(H_N)$. To see this, it suffices to consider the caps $C^p$ with $x=x_p \ge N^{2/3}$ in which case $C^p$ lies between the horizontal lines $y=n$ and $y=n+1$ with $n=1,\ldots,N^{1/3}$ so $x\asymp \frac Nn$. The corresponding egde $E_p$ connects the points $(\lceil \frac Nn \rceil,n)$ and $(\lceil \frac N{n+1} \rceil,n+1)$. So $\rho \asymp \frac N{n^2}$ and
\[ 
\Area C^p \asymp \frac {\rho^3}r \asymp \frac N{n^3}.
\]
The area of $\bigcup_{x_p\ge N^{2/3}} C^p$ is of order $\sum_1^{N^{1/3}}\frac N{n^3} \asymp N$.

\bigskip
\section{Proof of Theorem~\ref{th:area}, the upper bound} 

\medskip
The target is to show that $\Area \left(\bigcup C^p \right) \ll N^{1/3}\log N$ where the union is taken over all $p$ with $N^{1/2}\le x_p\le N_2 \asymp N^{2/3}$. For $j=1,2,\ldots, \lceil \frac {\log N}6 \rceil$ we define the interval
\[
I_j=[2^{-j}N_2,2^{-j+1}N_2]. 
\]

Let $G_j$ be the set of inner normals $p$ with $x_p \in I_j$. We are going to show that $ \Area \left(\bigcup_{p\in G_j}\right) C^p \ll N^{1/3}$ for every $j$. This will complete the proof as there are $O(\log N)$ such intervals. For a fixed $p \in G_j$, the cap $C^p$ has parameters $x_p,h_x,r_x$ and $\rho_x$. The advantage of the intervals $I_j$ is that the radius of curvature $r_x\asymp \frac {x^3}N$ on $I_j$ is almost constant, and so is the slope $-\frac N{x^2}$ of the tangent to $H^1$ at $x$:
\[
2^{-3j}N \ll r_x \ll 2^{-3j+3}N \mbox{ and }  -2^{2j}N^{-1/3} \le -\frac N{x^2} \le -2^{2j-2}N^{-1/3} 
\]
We set $R_j:= 2^{-3j}N$ and define $\Phi_j\in (0,\pi/4]$ via $\tan \Phi_j:= 2^{2j}N^{-1/3}$. So for all $x\in I_j$
\begin{equation}\label{eq:bounded}
x\asymp 2^{-j}N^{2/3}\mbox{ and }  r_x\asymp R_j \mbox{ and }  \frac N{x^2}\asymp \tan \Phi_j \mbox{ and }|I_j|=2^{-j}N^{2/3}.
\end{equation}
We distinguish three cases.

\medskip
{\bf Case 1} when $h_x \le r_x^{-1/3}$. Then $h_x\ll \left(\frac{x^3}N\right)^{-1/3}=\frac {N^{1/3}}x$. So in this case every cap has height $\ll \frac {N^{1/3}}x$ and the projection to the $x$ axis of the union of these caps has length $\ll |I_j|$, because $|z_2-z_1|\asymp \rho\asymp\sqrt{hr}\ll\sqrt r\ll |I_j|N^{1/6}x^{-1/2}$. So the area of the union in Case 1 is $\ll \frac {N^{1/3}}x |I_j| \ll N^{1/3}$.

\medskip
In the remaining cases we will use the flatness theorem. In the 2-dimensional case it says that if a convex set $K \subset \RR^2$ (with nonempty interior) is lattice point free, then its lattice width is at most $1+2\sqrt 3$. This is a result of Hurkens~\cite{Hur}. The original flatness theorem in all dimensions is due to Khintchine~\cite{Khin}, for a more recent version and with better constants, see \cite{KannL}. In our case each $C^p$ (or rather the interior of $C^p$) is lattice point free, so there is a primitive vector $q\in \ZZ^2$ such that 

\[
\max \{q(z_1-z_2): z_1,z_2 \in C^p\} \le 1+2\sqrt 3. 
\]
That is, the flatness of $C^p$ is at most $1+2\sqrt 3$ and its flatness direction is $q$.

\medskip
{\bf Case 2} when $p^{\perp}$ is the flatness direction of $C^p$. This implies $h_x\le \frac{1+2\sqrt 3}{|p|}$. 
Using (\ref{eq:sim}) gives $\rho^2\asymp hr \ll \frac r{|p|}$. The edge of $C^p$ has length at least $|p|$ (right?) and so $|p|\le 2\rho \ll \sqrt{\frac r{|p|}}$. 
Implying  $|p| \ll r^{1/3}$ that will be needed soon. Thus, using $r\asymp R_j$ from (\ref{eq:bounded})
\begin{align}\label{eq:p3/2}
 &\Area C^p\asymp h\rho \ll \frac{1+2\sqrt 3}{|p|}\sqrt{\frac{r}{|p|}}\ll \sqrt{\frac r{|p|^3}} \mbox{ and }\\ \nonumber
 &S_j:=\sum \Area C^p \ll \sqrt R_j \sum \frac 1{|p|^{3/2}}
\end{align}
where the last two sums are taken over all $p \in G_j$ for which $C^p$ is in Case 2. 

\medskip
The last sum in (\ref{eq:p3/2}) does not change if we replace $p=(a,-b)$ by $(a,b)$ so we can take it when $p$ runs over all primitive vectors $p=(a,b)$ with $|p|\ll R_j^{1/3}$ and the slope of $p$ is $\ll \tan \Phi_j$. 

\medskip
The last sum in (\ref{eq:p3/2}) can be estimated by (we drop the condition of $p$ being primitive)
\[
\ll \sum_{a\ll R_j^{1/3}} \frac1{a^{3/2}}\sum_{b\ll a\tan \Phi_j}1 \ll \tan\Phi_j\sum_{a\ll R_j^{1/3}}\frac1{a^{1/2}} \ll
\tan\Phi_j R_j^{1/6}.
\]

Using the values of $R_j$ and $\Phi_j\asymp \tan \Phi_j$ we have 
\[
S_j \ll R_j^{2/3}\frac {2^{2j}}{N^{1/3}} \ll N^{1/3}.
\]

\smallskip
{\bf Case 3} when $h_x>\frac{1+2\sqrt 3}{|p|}$ and, of course, $h_x>r_x^{-1/3}$. The flatness direction of $C^p$ is a primitive vector $q=(u,v) \in \ZZ^2$ different from $p^{\perp}$ and we associate $q$ with $p$ or $C^p$. Then 
\[
h_x \le \frac{1+2\sqrt 3}{|q|} \mbox{ implying }\, |q|\ll r_x^{1/3}.
\]
The line of $E_p$ intersects $H^1$ at points $z_1$ and $z_2$, see Figure~\ref{fig:pandq} where the hyperbola $H_N$ is not shown as it is very close to the osculating circle (which is coloured red). Claim~\ref{cl:rho} shows that $|z_2-z_1|\asymp 2\rho$ (although Figure~\ref{fig:pandq} apparently shows that $|z_2-z_1|=2\rho$).

\begin{figure}[h!]
\centering
\includegraphics[scale=0.8]{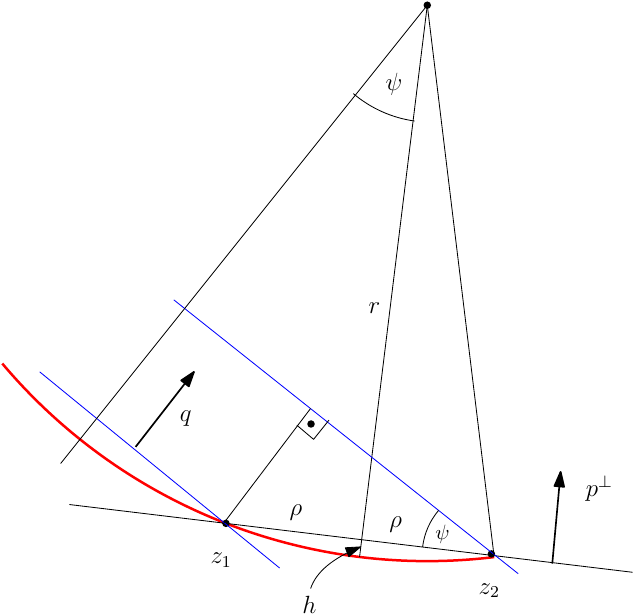}
\caption{The cap $C^p$ and the osculating circle.} 
\label{fig:pandq}
\end{figure}

\medskip
We determine the angle $\psi$ of $p^{\perp}$ and $q$. In the right angled triangle with vertices $z_1,z_2$ (see Figure~\ref{fig:pandq})
\[
|\sin \psi|= \frac{|(z_2-z_1)q|}{|z_2-z_1||q|}\ll \frac {1+2\sqrt 3}{2\rho |q|}\ll \frac 1{\sqrt{rh}|q|}\ll \frac 1{r^{1/3}|q|}\le r^{-1/3},
\]
because $\|q\|\geq1$. Thus, $\psi$ is small and $\tan(\Phi_j+\psi)\asymp \tan\Phi_j$. 

\medskip
Let $qx=k_0$ be the equation of the line $L_0$ tangent to $H^1$ (with normal $q$). It follows from equation (\ref{eq:sectons}) that $k_0=\sqrt{4uvN}$. The $x$-component of the point of $L_0\cap H^1$ is $x(q)=\frac {k_0}{2u}$ and $x(q)$ may not be in $I_j$ but $x(q)\in I_{j-1}\cup I_j\cup I_{j+1}$ because $\psi$ is very small compared to $\Phi_j$. Then 
\begin{equation}\label{eq:Ij}
|x(q)-x|\le 2|I_j|=2^{-j+1}N^{2/3}.
\end{equation}

\medskip
Let $qx=k_0+k_i$ be the equation of the line $L_i$ containing $z_i$, $i=1,2$. We record the inequality
\begin{equation}\label{eq:k1k2}
1\leq |(z_2-z_1)q|=k_2-k_1 \le 1+2\sqrt3, 
\end{equation}
where the upper bound follows from the fact that $q$ is the flatness direction of $C^p$ and the lower bound holds because there are at least two lattice points on the segment $[z_1,z_2]$, namely the endpoints of $E_p$. Let $z$ be the endpoint of $E_p$, the one closer to $z_2$, and let $qx=k_0+k$ be the equation of the line containing $z$. Then $k_0+k$ is a positive integer and $k_1,k,k_2$ are very close to each other since $k_1+1\le k\le k_2\le k_1+1+2\sqrt 3$. This time we associate the pair $(q,k_0+k)$ with the cap $C^p$. 
Note that $k>0$ and $k_0+k$ is a positive integer. In this way $(q,k_0+k)$ is associated with at most two caps because the line $qx=k_0+k$ contains at most two points from the boundary of $I(Q_N)$; when there are two such caps, one of them is to the left of $x(q)$, the other to the right.

\medskip
 $L_i$ intersects $H_1$ in two points. The difference between their $x$ components can be computed using (\ref{eq:sectons}). For $L_i$ this difference is 
\[
d_i:= \frac {\sqrt{(k_i+k_0)^2-k_0^2}}{u}=\frac {\sqrt{k_i^2+2k_ik_0}}{u}. 
\]
Observe that $2\rho  \asymp |z_2-z_1|\asymp d_2-d_1$. We estimate $d_2-d_1$:
\begin{align}\label{eq:2rho}
2\rho &\asymp\frac 1u \left({\sqrt{k_2^2+2k_2k_0}}-{\sqrt{k_1^2+2k_1k_0}}\right)=\frac1u \frac {(k_2^2+2k_2k_0)-(k_1^2+2k_1k_0)} {\sqrt{k_2^2+2k_2k_0}+\sqrt{k_1^2+2k_1k_0}}\\\nonumber 
   &\asymp\frac{k_2-k_1}u \frac{k_2+k_1+2k_0}{\sqrt{k_2^2+2k_2k_0}+\sqrt{k_1^2+2k_1k_0}}\\\nonumber
   &\asymp\frac1u \frac{k+k_0}{\sqrt{k^2+kk_0}}\asymp \frac {\sqrt{k+k_0}}{u\sqrt k},
\end{align}
the last line is justified by (\ref{eq:k1k2}) and by the fact that $k_1,k,k_2$ are close to each other. 

\begin{figure}[h!]
\centering
\includegraphics[scale=0.7]{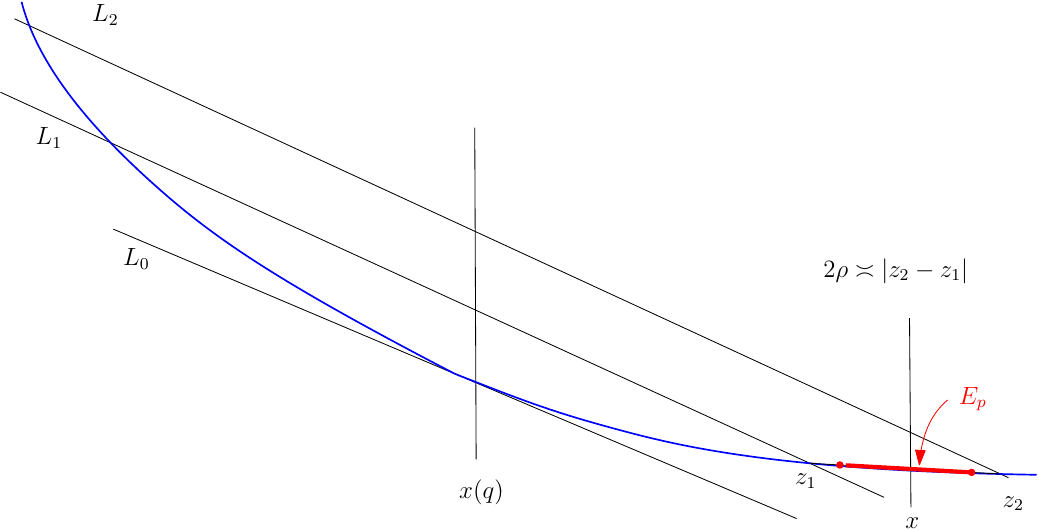}
\caption{The lines $L_i$ and $E_p$.} 
\label{fig:k1k2}
\end{figure}

\medskip
We {\bf claim} $k\asymp k_2 \ll k_0$. Observe that $d_2=\frac{\sqrt{k_2^2+2k_2k_0}}u\ge \frac {k_2}{u}$. Figure~\ref{fig:k1k2} shows that $d_2/2>2\rho$ and $|x-x(q)|\ge d_2/2-\rho\ge d_2/4$. Thus $|x-x(q)|\ge \frac{k_2}{4u}$. On the other hand $\frac{k_0}{4u}=\sqrt{\frac v{2u}N}\asymp 2^{-j}N^{2/3}$ because $\frac vu$ is the slope of $q$ which $\asymp 2^{-2j}N^{1/3}$. (\ref{eq:Ij}) shows that $|x-x(q)|\le 2|I_j|=2^{-j+1}N^{2/3}\asymp \frac{k_0}{4u}$. Comparing the lower and upper bounds on $|x-x(q)|$ yields $k_2 \ll k_0$ indeed.

\medskip
We continue equation (\ref{eq:2rho}) using $k\ll k_0$
\[
2\rho \asymp \frac {\sqrt{k_0}}{u\sqrt {k}}=\frac {\sqrt[4]{4uvN}}{u\sqrt k}\asymp \frac {\sqrt[4]{\left({\frac vu}\right)^3N}}{\sqrt{vk}}\asymp \frac{\sqrt {r}}{\sqrt{|q|k}},
\]
where it is easy to check that $r=r_x \asymp \sqrt{\left({\frac vu}\right)^3N} $

\medskip
From this we see that $h \ll \frac 1{k|q|}$ because $hr\asymp\rho^2$. Further $k|q|\ll \frac 1h \ll r^{1/3}\asymp R_j^{1/3}$. Then 
\[
\Area C^p \ll h\rho \ll \frac 1{k|q|} \sqrt{\frac r{k|q|}}\ll \frac {\sqrt R_j}{(k|q|)^{3/2}}.
\]
We note here that, for a fixed $q$, $k_0+k$ are different positive integers for all caps $C^p$ in Case 3 that are to the left of $x(q)$. Since and $k_0\le k_1$ and $k_1+1\le k$, we have
\[
\Area C^p \ll \frac {\sqrt R_j}{(\lfloor k \rfloor|q|)^{3/2}}
\]
and the positive integers $k^*:=\lfloor k \rfloor$ are all distinct for the caps to the left of $x(q)$. The same applies to the caps $C^p$ in Case 3 that are to the right of $x(q)$. So it suffices to consider the caps to the left of $x(q)$.

\medskip
The sum in the last formula corresponds to (\ref{eq:p3/2}) with $k^*|q|=|k^*q|$ instead of $|p|$ but with essentially the same conditions. The angle of $q=(u,v)$ differs from that of $p^{\perp}=(b,a)$ by at most $\arcsin \psi \le r^{-1/3}$. In the same way as in Case 2 we consider $\sum (k^*|q|)^{-3/2}$ for all primitive $q$ and integer $k^*$ with $k^*|q| \ll R_j^{1/3}$ but instead of $q=(u,v)$ we sum for $q'=(v,u)$ that spans a positive angle with the $x$-axis, as small as $\ll \arcsin \Phi_j+\arcsin \psi\ll \arcsin \Phi_j \ll \frac N{x^2}$. This means that the slope of $q'$ is $\ll \tan\Phi_j$ again, and so the contribution of all $k^*q'$ is already contained in the computation in Case 2. For $k^*=1$ they are the primitive $p$s and for $k^*>1$ they are the non-primitive ones there.

\medskip
In every case the area of the union of the caps is $\ll N^{1/3}$, so indeed $S_j \ll N^{1/3}$\qed
\bigskip
\section{Proof of Theorem~\ref{th:area}, the lower bound} 

\medskip
The lower bound for $A_N$ is simple if we use Theorem~\ref{th:main2}. As before, the primitive vector $p=(a,-b) \in \ZZ^2$ with $0<b\le a$ is the direction of an edge $E_p$ of $I(H_N)$, and $L_p$, the line of $E_p$, has the equation $bx+ay=k$ for some integer $k$. The equation of the tangent line to $H^1$ parallel to $L_p$ is $bx+ay=\sqrt{4abN}=\kappa$. The cap $T^p$ is defined by translating the tangent line up to its intersection with $H_N$ to have length $|p|$. The equation of the translated line is $bx+ay=\lambda$ for some real number $\lambda$. As there are no lattice points, not alone a unit square, in the open cap cut from $H_N$ by $L^p$, we easily have $\kappa<\lambda < 2\kappa$ and actually $\kappa=\lceil\lambda\rceil$.

The area between $E_p$ and $H^1$ (cut by vertical lines) is $\ge \Area T^p$, hence
\begin{equation}\label{eq:AN}
A_N\ge2\sum_{a,b}\Area T^p,
\end{equation}
where the sum is taken over the slopes of the edges $E_p$ with horizontal projection in $[N^{1/2},N^{2/3}]$. The computations around equations (\ref{eq:sectons}) and (\ref{eq:tang}) show that the horizontal projection of $T^p$ is $\sqrt{\lambda^2-4abN}/b$ and this should be $a$ by definition, and so 
\[
\lambda^2-\kappa^2=(ab)^2, \mbox{ or }\, \lambda - \kappa =\frac{(ab)^2}{\lambda+\kappa} \asymp \frac{(ab)^2}{\kappa} = \frac{(ab)^{3/2}}{\sqrt {4N}}.
\]
The area of $T^p$ can be trivially estimated from below by the area of the triangle with vertical height $(\lambda-\kappa)/a$ and horizontal projection $a$, which means in (\ref{eq:AN})
\[
A_N\ge2\sum_{a,b}\Area T^p\ge\sum_{a,b} (\lambda-\kappa)\gg\sum_{a,b}\frac{(ab)^{3/2}}{N^{1/2}}.
\]
This sum has $R\asymp N^{1/3}\log N$ terms (see Theorem~\ref{th:main2}) and is smallest when $(a,b),\,1\le b\le a$ runs over different primitive vectors with smallest possible $ab$, say $ab\le w$. The analysis is exactly the same as in the calculation of (\ref{eq:3root}) with $w\asymp R/\log R\asymp N^{1/3}$, $w\log w\asymp R\asymp N^{1/3}\log N$. \qed

\vskip1cm
{\bf Acknowledgements.} The second author (IB) was partially supported by NKFIH grant No. 133819 and also by the HUN-REN Research Network.

\end{document}